\documentclass[12pt, reqno]{amsart}

\usepackage{amsthm, amsmath, amsfonts, amssymb, graphicx, tikz, calrsfs}
\usetikzlibrary{cd}

\newtheorem{theorem}{Theorem}[section]
\newtheorem{lemma}[theorem]{Lemma}
\newtheorem{proposition}[theorem]{Proposition}
\newtheorem{corollary}[theorem]{Corollary}

\theoremstyle{definition}
\newtheorem{definition}[theorem]{Definition}
\newtheorem{example}[theorem]{Example}

\newtheorem{remark}[theorem]{Remark}
\newtheorem{notation}[theorem]{Notation}

\newcommand{\func}[1]{\operatorname{#1}}

\newcommand{\Op}{\mathcal{O}}
\newcommand{\down}{{\downarrow}}
\newcommand{\up}{{\uparrow}}
\newcommand{\cl}{{\sf cl}}
\newcommand{\Int}{{\sf int}}

\newcommand{\pt}{{\sf pt}}
\newcommand{\RC}{{\sf RC}}
\newcommand{\U}{\mathcal{U}}
\newcommand{\X}{\mathcal{X}}
\newcommand{\F}{\mathcal{F}}

\begin{document}

\title{When is the frame of nuclei spatial: A new approach}
\author{F. \'{A}vila, G.~Bezhanishvili, P.~J.~Morandi, A.~Zald{\'i}var}
\date{}

\dedicatory{Dedicated to the memory of Harold Simmons}

\begin{abstract}
For a frame $L$, let $X_L$ be the Esakia space of $L$. We identify a special subset $Y_L$ of $X_L$ consisting of nuclear points of
$X_L$, and prove the following results:
\begin{itemize}
\item $L$ is spatial iff $Y_L$ is dense in $X_L$.
\item If $L$ is spatial, then $N(L)$ is spatial iff $Y_L$ is weakly scattered.
\item If $L$ is spatial, then $N(L)$ is boolean iff $Y_L$ is scattered.
\end{itemize}
As a consequence, we derive the well-known results of Beazer and Macnab \cite{BM79}, Simmons \cite{Sim80}, Niefield and Rosenthal
\cite{NR87}, and Isbell \cite{Isb91}.
\end{abstract}

\maketitle

\section{Introduction}

Nuclei play an important role in pointfree topology as they are in 1-1 correspondence with onto frame homomorphisms, and hence describe
sublocales of locales \cite{DP66,Isb72}. For each frame $L$, let $N(L)$ be the set of nuclei on $L$. There is a natural order on $N(L)$
given by
\[
j\le k \mbox{ iff } ja\le ka \mbox{ for each } a\in L.
\]
With this order $N(L)$ is also a frame \cite{Isb72,BM79,Joh82}, which we will refer to as the \emph{frame of nuclei} or the
\emph{assembly} of $L$. The complicated structure of $N(L)$ has been investigated by many authors;
see for example \cite{DP66, Isb72, Sim78, BM79, Sim80, Mac81, Joh82, NR87, Isb91, Wil94, Ple00, Ple02, BG07, BGJ13, Sim14,BGJ16}.

To describe some of the landmark results about $N(L)$, we recall that a frame $L$ is \emph{spatial} if it is isomorphic to the frame
$\Op S$ of open sets of a topological space $S$.
For a subspace $T$ of $S$, a point $x \in T$ is \emph{isolated} in $T$ if $\{x\}$ is an open subset of $T$, and
\emph{weakly isolated} in $T$ if there is an open subset $U$ of $S$ such that $x \in T \cap U \subseteq \overline{\{x\}}$.
It is clear that an isolated point of $T$ is weakly isolated in $T$.

The space $S$ is \emph{scattered} if each nonempty subspace of $S$ contains an isolated point. It is easy to see that $S$ is scattered
iff each nonempty closed subspace of $S$ has an isolated point.
The space $S$ is \emph{weakly scattered} if each nonempty closed subspace has a weakly isolated point.
It is well known that $S$ is scattered iff $S$ is weakly scattered and $T_D$, where $S$ is
$T_D$ if each point is the intersection of an open and a closed set (see, e.g., \cite[Sec.~VI.8.1]{PP12}).

We next describe some of the landmark results about $N(L)$.
\begin{itemize}
\item Beazer and Macnab \cite{BM79} proved that if $L$ is boolean, then $N(L)$ is isomorphic to $L$, and gave a necessary and sufficient
condition for $N(L)$ to be boolean.
\item Simmons \cite{Sim80} proved that if $S$ is a $T_0$-space, then $N(\Op S)$ is boolean iff $S$ is scattered; and that dropping the $T_0$
assumption results in the following more general statement: $N(\Op S)$ is boolean iff $S$ is dispersed (see Section~\ref{sec: Simmons} for the definition).
\item Simmons \cite[Thm.~4.4]{Sim80} also gave a necessary and sufficient condition
for $S$ to be weakly scattered. This result of Simmons is sometimes stated erroneously as follows: $N(\Op S)$ is spatial iff $S$ is weakly
scattered (see, e.g., \cite[p.~267]{NR87}). While this formulation is false (see Example~\ref{ex: 7.6}), Isbell \cite{Isb91}  proved that if
$S$ is a sober space, then indeed $N(\Op S)$ is spatial iff $S$ is weakly scattered.
\item Niefield and Rosenthal \cite{NR87} gave necessary and sufficient conditions for $N(L)$ to be spatial, and derived that if $N(L)$ is
spatial, then so is $L$.
\end{itemize}

Nuclei on $L$ can be studied by utilizing Priestley duality \cite{Pri70} for distributive lattices and
Esakia duality \cite{Esa74} for Heyting algebras. This was done independently in \cite{PS00} and \cite{BG07}. While the authors of \cite{PS00} did not use Esakia duality, it was utilized in \cite{BG07} where it was shown that nuclei on a Heyting algebra $L$ correspond to special closed subsets
of the Esakia space $X_L$ of $L$ (see Section~\ref{sec: nuclear} for the definition). In this paper we term such closed subsets \emph{nuclear}. In \cite{PS00} these were called $L$-sets.
If $N(X_L)$ denotes all nuclear subsets of $X_L$, then we utilize the dual isomorphism between $N(L)$ and $N(X_L)$ to give an
alternate proof of the results mentioned in the previous paragraph. We single out a subset $Y_L$ of $X_L$ consisting of nuclear points of $X_L$ and show that $L$
is spatial iff $Y_L$ is dense in $X_L$ (see also \cite[Sec.~2.11]{PS00}). We prove that join-prime elements of $N(X_L)$ are exactly the singletons $\{y\}$ where $y\in Y_L$.
From this we derive that there is a bijection between the points of $L$ and the points of $N(L)$. We also obtain a characterization of when
$N(L)$ is spatial in terms of $X_L$, which yields an alternate proof of the results of Niefield and Rosenthal \cite{NR87}.

We prove that $N(L)$ is boolean iff the set of maximal points of each clopen downset of $X_L$ is clopen. From this we derive an alternate
proof of the result of Beazer and Macnab \cite{BM79}. We next turn to the setting of $L=\Op S$ for some topological space $S$. We show
that $Y_{\Op S}$ is homeomorphic to the soberification of $S$, and prove that $N(\Op S)$ is spatial iff $Y_{\Op S}$ is weakly scattered.
This implies that $N(\Op S)$ is spatial iff the soberification of $S$ is weakly scattered. As a corollary we obtain the result of Isbell
\cite{Isb91} that if $S$ is sober, then $N(\Op S)$ is spatial iff $S$ is weakly scattered. We give an example showing that this result
is false if $S$ is not assumed to be sober.

We finally turn to the results of Simmons \cite{Sim80}. One of Simmons' main tools is the use of the front topology. We show that if
$S$ is $T_0$, then $X_L$ is a compactification of $S$ with respect to the front topology on $S$. From this, by utilizing the
$T_0$-reflection, we derive Simmons' characterization \cite[Thm.~4.4]{Sim80} of arbitrary (not necessarily $T_0$) weakly scattered spaces.
In addition, we prove that $N(L)$ is boolean iff $Y_L$ is scattered. From this we derive Simmons' theorem that if $S$ is $T_0$, then
$N(\Op S)$ is boolean iff $S$ is scattered. We generalize this result to an arbitrary
space by showing that $S$ is dispersed iff its $T_0$-reflection is scattered. This yields the general form
of Simmons' theorem that $N(\Op S)$ is boolean iff $S$ is dispersed.

\section{Frames, spaces, and nuclei} \label{sec: frames, spaces, nuclei}

In this section we recall basic facts about frames. We use \cite{Sim78,Joh82,PP12} as our basic references. A \emph{frame} is a complete lattice
$L$ satisfying the join-infinite distributive law
\[
a \wedge \bigvee S = \bigvee \{ a\wedge s \mid s \in S\}.
\]
Frames are complete Heyting algebras where the implication is defined by
\[
a \to b = \bigvee \{ x \in L \mid a \wedge x \le b\}.
\]
We set $\lnot a = a \to 0$.

A \emph{frame homomorphism} is a map $h:L \to K$ preserving finite meets and arbitrary joins.
As usual, we denote by $\sf Frm$ the category of frames and frame homomorphisms.
Let also $\sf Top$ be the category of topological space and continuous maps. There is a contravariant functor $\Op:{\sf Top}\to{\sf Frm}$
sending each topological space $S$ to the frame $\Op S$ of opens of $S$, and each continuous map $f:S\to T$ to the frame homomorphism
$f^{-1}:\Op T\to\Op S$.

To define a contravariant functor in the other direction, we recall that a \emph{point} of a frame $L$ is a frame homomorphism $p:L\to{\sf 2}$
where ${\sf 2}=\{0,1\}$ is the two-element frame. It is well known that there is a one-to-one correspondence between points of $L$, meet-prime
elements of $L$, and completely prime filters of $L$ (see, e.g., \cite[Sec.~II.1.3]{Joh82}).
We will mainly think of points as completely prime filters of $L$, but at times it will also be convenient to think of them as meet-prime
elements of $L$.

Let $\pt(L)$ be the set of points of $L$. For $a\in L$, we set
\[
\eta(a) = \{ x \in \pt(L) \mid a\in x\}.
\]
Then $\{\eta(a)\mid a\in L\}$ is a topology on $\pt(L)$, and $\eta:L\to\Op(\pt L)$ is an onto frame homomorphism. If $h:L\to K$ is a
frame homomorphism, then $\pt(h):\pt(K)\to\pt(L)$ given by $\pt(h)(y)=h^{-1}(y)$ is a continuous map. This defines a contravariant functor
$\pt:{\sf Frm}\to{\sf Top}$.

The functors $\pt,\Op$ yield a contravariant adjunction between $\sf Frm$ and $\sf Top$. The unit of the adjunction is given by the
frame homomorphism $\eta:L\to\Op(\pt L)$, and the counit by the continuous map $\varepsilon:S\to\pt(\Op S)$ where
$\varepsilon(s)=\{U\in\Op S\mid s\in U\}$.

We call a frame $L$ \emph{spatial} if $\eta$ is an isomorphism and a space $S$ \emph{sober} if $\varepsilon$ is a homeomorphism.
It is well known that $L$ is spatial iff whenever $a,b \in L$ with $a \not\le b$, there is a point $x$ with $a \in x$ and $b \notin x$ (see, e.g., \cite[Sec.~II.1.5]{Joh82}); that $S$ is sober iff each irreducible closed set is the closure of a unique point (see, e.g., \cite[Sec.~II.1.6]{Joh82}); and that the contravariant adjunction $(\pt,\Op)$ restricts to a dual equivalence between the category $\sf SFrm$ of spatial frames and
the category $\sf Sob$ of sober spaces (see, e.g., \cite[Sec.~II.1.7]{Joh82}).

\begin{definition} \cite[Def.~1]{Sim78}
A nucleus on a frame $L$ is a map $j : L \to L$ satisfying
\begin{enumerate}
\item $a \le ja$;
\item $jja \le ja$;
\item $j(a\wedge b) = ja \wedge jb$.
\end{enumerate}
\end{definition}

Nuclei play an important role in pointfree topology as they characterize quotients of frames: If $h:L\to K$ is a frame homomorphism and
$r : K \to L$ is its right adjoint, then $rh$ is a nucleus on $L$; conversely, if $j$ is a nucleus on $L$, then the fixpoints
\[
L_j := \{ a\in L \mid a=ja \} = \{ ja \mid a\in L \}
\]
form a frame where finite meets are the same as in $L$ and the joins are defined by
\[
\bigsqcup S=j\left(\bigvee S\right)
\]
for each $S\subseteq L_j$. This establishes a one-to-one correspondence between onto frame homomorphisms and nuclei on $L$ (see, e.g., \cite[Prop.~III.5.3.2]{PP12}).

Let $N(L)$ be the set of all nuclei on $L$. Define a partial order $\le$ on $N(L)$ by
\[
j\le k \mbox{ iff } ja\le ka \mbox{ for each } a\in L.
\]
As we pointed out in the introduction, it is well known that $N(L)$ is a frame with respect to $\le$. Finite meets are
defined in $N(L)$ componentwise, the bottom $\bot$ is the identity nucleus, and the top $\top$ is the nucleus sending
every element of $L$ to $1$. Calculating joins in $N(L)$ is more involved (see, e.g., \cite[Sec.~II.2.5]{Joh82}). The frame $N(L)$ is often
referred to as the \emph{assembly} of $L$.

The following nuclei play an important role:
\begin{align*}
u_a(x) &= a\vee x;\\
v_a(x) &= a \to x;\\
w_a(x) &= (x \to a) \to a.
\end{align*}
It is well known (see, e.g., \cite[Lem.~7(ii)]{Sim78}) that each nucleus $j$ can be written as
\[
j=\bigwedge \{ w_a \mid a = ja \} = \bigwedge \{ w_{ja} \mid a \in L\}.
\]
Moreover, sending $a\in L$ to $u_a$ defines a frame embedding of $L$ into $N(L)$, and we can form the \emph{tower} of assemblies:
\[
L \hookrightarrow N(L) \hookrightarrow N^2(L) \hookrightarrow \cdots
\]

The \emph{booleanization} of a frame $L$ is defined as the fixpoints of $w_0$:
\[
B(L) := \{ w_0(a) \mid a \in L\} = \{ \lnot \lnot a \mid a \in L\}.
\]
It is well known that $B(L)$ is a boolean frame (a complete boolean algebra). In fact, the embedding $L\hookrightarrow N(L)$ factors through $B(N(L))$ since $u_a,v_a$ are complemented elements of $N(L)$, hence
belong to $B(L)$ (see, e.g., \cite[Sec.~II.2.6]{Joh82}).

\section{Priestley and Esakia dualities} \label{sec: dualities}

In this section we recall Priestley duality for bounded distributive lattices, and Esakia duality for Heyting algebras. We use
\cite{Pri70,Pri72,Esa74} as our basic references. A subset of a topological space $X$ is \emph{clopen} if it is both closed and
open. If $\le$ is a partial order on $X$ and $S\subseteq X$, then
\[
{\uparrow}S := \{ x\in X \mid s\le x \mbox{ for some } s\in S\} \mbox{ and }
{\downarrow}S := \{ x\in X \mid x\le s \mbox{ for some } s\in S\}.
\]
If $S=\{s\}$, then we simply write ${\uparrow}s$ and ${\downarrow}s$. We call $S$ an \emph{upset} if $S={\uparrow}S$, and a
\emph{downset} if $S={\downarrow}S$.

\begin{definition} \cite{Pri70}
A \emph{Priestley space} is a pair $(X,\le)$ where $X$ is a compact space, $\le$ is a partial order on $X$, and
the \emph{Priestley separation axiom} holds:
\[
\mbox{ From } x \not\le y \mbox{ it follows that there is a clopen upset } U \mbox{ containing } x \mbox{ and missing } y.
\]
\end{definition}

Let $\sf Pries$ be the category of Priestley spaces and continuous order preserving maps, and let $\sf Dist$ be the category of
bounded distributive lattices and bounded lattice homomorphisms.

\begin{theorem} [Priestley duality]
$\sf Pries$ is dually equivalent to $\sf Dist$.
\end{theorem}

The contravariant functor $\X:{\sf Dist}\to{\sf Pries}$ sends a bounded distributive lattice $L$ to the Priestley space
$X_L$ of prime filters of $L$ ordered by inclusion. The topology on $X_L$ is given by the basis
\[
\{\varphi(a) \setminus \varphi(b)\mid a,b\in L\}
\]
where
\[
\varphi(a) = \{ x \in X_L \mid a \in x\}.
\]
If $h:L\to K$ is a bounded lattice homomorphism, then $\X(h):X_K\to X_L$ is given by $\X(h)(x)=h^{-1}[x]$.

The contravariant functor $\U:{\sf Pries}\to{\sf Dist}$ sends a Priestley space $X$ to the bounded distributive lattice
$\U(X)$ of clopen upsets of $X$, and a morphism $f:X\to Y$ to the bounded lattice homomorphism
$\U(f):\U(Y)\to\U(X)$ given by $\U(f)(U)=f^{-1}(U)$.

The unit of this dual equivalence is given by the isomorphism $\varphi:L\to\U(X_L)$ in $\sf Dist$, and the counit by the
isomorphism $\xi:X\to X_{\U(X)}$ in $\sf Pries$ given by
\[
\xi(x)=\{U\in\U(X)\mid x\in U\}.
\]

\begin{notation} \label{notation}
For a Priestley space $X$, we denote by $\pi$ the topology on $X$, by $\pi_u$ the topology of open upsets, and by $\pi_d$ the topology
of open downsets. We then have that $\pi=\pi_u\vee\pi_d$. We use $\cl_\pi$ and $\Int_\pi$ to denote the closure and interior in $(X, \pi)$.
\end{notation}

\begin{definition} \cite{Esa74}
A Priestley space $X$ is an \emph{Esakia space} if $U$ clopen in $X$ implies that ${\downarrow}U$ is clopen.
\end{definition}

Let $\sf Esa$ be the category of Esakia spaces and continuous maps $f:X\to Y$ satisfying ${\uparrow}f(x)=f[{\uparrow}x]$. Such maps are
known as \emph{bounded morphisms} or \emph{$p$-morphisms}. Each such map is order preserving, thus $\sf Esa$ is a non-full subcategory
of $\sf Pries$. Let $\sf Heyt$ be the category of Heyting algebras and Heyting algebra homomorphisms. Then $\sf Heyt$ is a non-full
subcategory of $\sf Dist$.

\begin{theorem} [Esakia duality]
$\sf Esa$ is dually equivalent to $\sf Heyt$.
\end{theorem}

The dual equivalence is established by the same functors $\mathcal{X}$ and $\mathcal{U}$. The additional condition on Esakia spaces
guarantees that $\U(X)$ is a Heyting algebra, where for $U,V\in\U(X)$,
\[
U \to V = X \setminus {\downarrow} (U \setminus V).
\]
Then the unit $\varphi:L\to\U(X_L)$ is a Heyting algebra isomorphism, and so for $a,b\in L$,
\[
\varphi(a\to b) = X_L \setminus {\downarrow} (\varphi(a) \setminus \varphi(b)).
\]
Also, being a bounded morphism yields that $\U(f)$ is a Heyting algebra homomorphism.

Since frames are complete Heyting algebras, their dual Esakia spaces satisfy an additional condition, which is an order-topological
version of extremal disconnectedness.

\begin{definition}
An Esakia space is \emph{extremally order-disconnected} if the closure of each open upset is clopen.
\end{definition}

\begin{theorem}  $($\cite[Thm.~2.3]{PS88}, \cite[Thm.~2.4(2)]{BB08}$)$
A Heyting algebra $L$ is a frame iff the Esakia space $X_L$ is extremally order-disconnected.
\end{theorem}

We will mostly work with extremally order-disconnected Esakia spaces since they are Esakia spaces of frames. We will frequently use the following well-known facts about Esakia spaces (see, e.g., \cite{Esa85}).

\begin{lemma} \label{lem: basic facts}
Let $X$ be an Esakia space.
\begin{enumerate}
\item The order $\le$ is closed in the product $X\times X$, so ${\uparrow}F,{\downarrow}F$ are closed for each
closed $F\subseteq X$.
\item If $F$ is a closed upset and $D$ a closed downset of $X$ with $F \cap D = \varnothing$, then there is a clopen upset $U$ with $F \subseteq U$ and $U \cap D = \varnothing$.
\item Each open upset is a union of clopen upsets and each open downset is a union of clopen downsets.
\item Let $F$ be a closed subset of $X$ and let $\max(F)$ be the set of maximal points of $F$. Then $\max(F) = \max(\down F)$ and $F \subseteq {\downarrow}\max(F)$.
\item If $F$ is a closed subset of $X$, then $\max(F)$ is closed.
\end{enumerate}
\end{lemma}

We point out that the first four items of Lemma~\ref{lem: basic facts} hold for any Priestley space.

\section{Nuclear subsets of Esakia spaces} \label{sec: nuclear}

Let $L$ be a Heyting algebra and $X_L$ its Esakia space. It was shown in \cite{BG07} that nuclei on $L$ are characterized as special
closed subsets of $X_L$. Further results in this direction were obtained in \cite{BGJ13,BGJ16}.

\begin{definition}
Let $X$ be an Esakia space.
\begin{enumerate}
\item We call a closed subset $F$ of $X$ a \emph{nuclear subset} provided for each clopen set $U$ in $X$, the set $\down(U \cap F)$
is clopen in $X$.
\item We call $x \in X$ a \emph{nuclear point} if $\{x\}$ is a nuclear subset of $X$.
\item Let $N(X)$ be the set of all nuclear subsets of $X$ ordered by inclusion.
\end{enumerate}
\end{definition}

\begin{remark}
In \cite{BG07} nuclear subsets of $X$ were called \emph{subframes} of $X$ because of their connection to subframe logics. In \cite{PS00} nuclear sets were called $L$-sets and nuclear points $L$-points. For our purposes the adjective nuclear appears more appropriate.
\end{remark}

\begin{theorem} \cite[Thm.~30]{BG07}
Let $L$ be a Heyting algebra and $X_L$ its Esakia space. Then $N(L)$ is dually isomorphic to $N(X_L)$.
\end{theorem}

The dual isomorphism is obtained as follows: If $j$ is a nucleus on $L$, then
\[
N_j := \{ x\in X_L \mid j^{-1}(x)=x \}
\]
is a nuclear subset of $X_L$; if $N$ is a nuclear subset of $X_L$, then $j_N:L\to L$ given by
\[
\varphi(j_N a) = X_L \setminus \down(N \setminus \varphi(a))
\]
is a nucleus on $L$; and this correspondence is a dual isomorphism. Therefore, if $L$ is a frame, and hence $X_L$ is an extremally
order-disconnected Esakia space, then $N(X_L)$ is a \emph{coframe}; that is, a complete lattice satisfying the meet-infinite distributive
law
\[
a\vee\bigwedge S=\bigwedge\{a\vee s\mid s\in S\}.
\]

Under the dual isomorphism, the bottom of $N(L)$ corresponds to $X_L$, the top of $N(L)$ to $\varnothing$, and $N_{j\wedge k}=N_j \cup N_k$
\cite[Cor.~31]{BG07}. Moreover, we have:

\begin{theorem} \cite[Thm.~34]{BG07} \label{thm: N_u}
\begin{enumerate}
\item $N_{u_a} = X_L \setminus \varphi(a);$
\item $N_{v_a} = \varphi(a);$
\item $N_{w_a} = \max(X_L \setminus\varphi(a)).$
\end{enumerate}
\label{thm: w_a}
\end{theorem}

The following is a quick corollary of Theorem~\ref{thm: N_u}(3).

\begin{corollary} \label{cor: max}
Let  $X$ be an Esakia space.
\begin{enumerate}
\item If $D$ is a clopen downset of $X$, then $\max(D) \in N(X)$.
\item If $U$ is a clopen subset of $X$, then $\max(U) \in N(X)$.
\item If $N \in N(X)$, then $\max(N) \in N(X)$.
\end{enumerate}
\end{corollary}

\begin{proof}
(1). This is an immediate consequence of Theorem~\ref{thm: N_u}(3).

(2). This follows from (1) since $\max(U) = \max(\down U)$ (see Lemma~\ref{lem: basic facts}(4)).

(3). This follows from (1) since $\down N$ is a clopen downset.
\end{proof}

We will also use the following basic facts about nuclear sets of extremally order-disconnected Esakia spaces.

\begin{lemma} \label{lem: basic nuclear facts}
Let $X$ be an extremally order-disconnected Esakia space.
\begin{enumerate}
\item If $U$ is a clopen subset of $X$, then $U \in N(X)$.
\item If $F$ is a regular closed subset of $X$, then $F \in N(X)$.
\item If $U$ is clopen and $N$ is nuclear, then $U \cap N \in N(X)$.
\end{enumerate}
\end{lemma}

\begin{proof}
(1). This is immediate since $X$ is an Esakia space.

(2). See \cite[Prop.~4.11]{BGJ13}.

(3). Let $V$ be clopen in $X$. Then $V \cap U$ is clopen, so $\down(V \cap U \cap N)$ is clopen since $N \in N(X)$. Thus, $U \cap N \in N(X)$.
\end{proof}

In contrast to Lemma~\ref{lem: basic nuclear facts}(3), it is not the case that the intersection of two nuclear sets is nuclear as the following simple example shows.
\begin{example}
Let $\mathbb{N}$ be the set of natural numbers. We view $\mathbb{N}$ as a discrete space, and let $X$ be the one-point compactification of $\mathbb{N}$ with order as in the following picture.
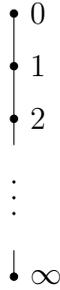
\begin{center}\begin{tikzpicture} [scale=0.7]
\draw [fill] (3,5) circle[radius=.07];
\draw [fill] (3,4) circle[radius=.07];
\draw [fill] (3,3) circle[radius=.07];
\draw [fill] (3,1.8) circle[radius=.02];
\draw [fill] (3,1.5) circle[radius=.02];
\draw [fill] (3,1.2) circle[radius=.02];
\draw [fill] (3,0) circle[radius=.07];
\node [right] at (3.1,5) {$0$};
\node [right] at (3.1,4) {$1$};
\node [right] at (3.1,3) {$2$};
\node [right] at (3.1,0) {$\infty$};
\draw (3,5) -- (3,2.5);
\draw (3,0) -- (3, .5);
\end{tikzpicture}
\end{center}
If $A$ and $B$ are the sets of even and odd numbers in $\mathbb{N}$, respectively, it is easy to see that $A \cup \{\infty\}$ and $B \cup \{\infty\}$ are both nuclear sets, but their intersection $\{\infty\}$ is not nuclear.
\end{example}

We conclude this
section by showing how to calculate meets in $N(X_L)$. This provides a dual description of joins in $N(L)$. For this we require the following lemma (see \cite[Prop.~2.5]{PS00}). 

\begin{lemma} \label{lem: nuclear}
Let $X$ be an extremally order-disconnected Esakia space and let $\{ N_\alpha \mid \alpha \in \Gamma\}$ be a family of nuclear subsets
of $X$. Then $\cl_\pi\left(\bigcup \{ N_\alpha \mid \alpha \in \Gamma \}\right)$ is a nuclear subset of $X$.
\end{lemma}

\begin{proof}
Let $F = \cl_\pi\left(\bigcup \{ N_\alpha \mid \alpha \in \Gamma \}\right)$. Then $F$ is closed. Since $X$ is a Priestley space, $\down F$
is closed. Therefore, $U := X \setminus \down F$ is an open upset. As $X$ is an extremally order-disconnected Esakia space,
$\cl_\pi(U)$ is a clopen upset. We show that $U = \cl_\pi(U)$. Clearly $N_\alpha \cap U = \varnothing$ for each $\alpha$, so
$\down N_\alpha \cap U = \varnothing$ since $U$ is an upset. Because $N_\alpha$ is nuclear, $\down N_\alpha$ is clopen, and hence
$\down N_\alpha \cap \cl_\pi(U) = \varnothing$. Thus, $N_\alpha \cap \cl_\pi(U) = \varnothing$ for each $\alpha$, and so
$F \cap \cl_\pi(U) = \varnothing$. Finally, as $\cl_\pi(U)$ is an upset, $\down F \cap \cl_\pi(U) = \varnothing$. This proves that
$U = \cl_\pi(U)$, so $U$ is clopen, and hence $\down F$ is clopen.

Now, let $V$ be a clopen subset of $X$. Then
\[
V \cap F = V \cap \cl_\pi\left(\bigcup \{ N_\alpha \mid \alpha \in \Gamma \}\right) =
\cl_\pi\left(\bigcup \{ V \cap N_\alpha \mid \alpha \in \Gamma\}\right).
\]
Since each $V \cap N_\alpha$ is nuclear by Lemma~\ref{lem: basic nuclear facts}(3), $\down (V\cap F)$ is clopen by the previous argument. Thus, $F$ is a nuclear subset of $X$.
\end{proof}

\begin{theorem} \label{prop: meet in NX}
Let $X$ be an extremally order-disconnected Esakia space, $\{N_\alpha \mid \alpha \in \Gamma \}$ a family of nuclear subsets of $X$,
and $N=\bigcap \{N_\alpha \mid \alpha \in \Gamma \}$. Then the meet of $\{N_\alpha \mid \alpha \in \Gamma \}$ in $N(X)$ is calculated
by the formula$:$
\[
\bigwedge \{N_\alpha \mid \alpha \in \Gamma\} = \cl_\pi\left( \bigcup \left\{ F \in N(X) \mid F \subseteq N \right\} \right).
\]
\end{theorem}

\begin{proof}
By Lemma~\ref{lem: nuclear}, $\cl_\pi\left( \bigcup \left\{ F \in N(X) \mid F \subseteq N \right\} \right)$ is a nuclear subset of $X$,
hence it follows from the definition that $\cl_\pi\left( \bigcup \left\{ F \in N(X) \mid F \subseteq N \right\} \right)$ is the greatest
lower bound of $\{N_\alpha \mid \alpha \in \Gamma \}$ in $N(X)$.
Thus, $\cl_\pi\left( \bigcup \left\{ F \in N(X) \mid F \subseteq N \right\} \right)=\bigwedge \{N_\alpha \mid \alpha \in \Gamma\}$.
\end{proof}

\begin{corollary}
Let $L$ be a frame, $X_L$ its Esakia space, and $\{ j_\alpha \mid \alpha \in \Gamma \}$ a family of nuclei on $L$. Then
\[
N_{\bigvee j_\alpha} = \bigwedge N_{j_\alpha}.
\]
\end{corollary}

\section{When are $L$ and $N(L)$ spatial?} \label{sec: spatial}

Let $L$ be a frame and let $X_L$ be its Esakia space. In this section we characterize in terms of $X_L$ when $L$ and $N(L)$ are spatial, and show how to derive the results of Niefield and Rosenthal \cite{NR87} from our characterizations.
For this we require the following definition.

\begin{definition}
For a frame $L$, let $Y_L = \{ y \in X_L \mid \{y\} \in N(X_L) \}$. Thus, $Y_L$ is the set of nuclear points of $X_L$.
\end{definition}

\begin{remark}
By \cite[Lem.~5.1]{BGJ16}, $y\in Y_L$ iff $y$ is a completely prime filter of $L$. Thus, points of $Y_L$ correspond to points of $L$.
As we will see in Proposition~\ref{prop: YL = pt(L)}, this correspondence is a homeomorphism.
\end{remark}

We recall that the \emph{specialization order} of a topological space $S$ is defined by
\[
x\le y \mbox{ iff } x\in\overline{\{y\}},
\]
where $\overline{\{y\}}$ denotes the closure of $\{y\}$ in $S$. Then it is easy to see that $\down y = \overline{\{y\}}$.
We also recall from Section~\ref{sec: dualities} that $\pi$ is the original topology on $X_L$, and $\pi_u$ is the topology of open upsets.
To simplify notation, we use the same $\pi$ for the subspace topology on $Y_L$, and let
$\tau$ be the subspace topology on $Y_L$ arising from $\pi_u$. We use $\cl_\tau$ for closure in $(Y_L, \tau)$.

\begin{lemma} \label{lem: tau}
Let $L$ be a frame and $X_L$ its Esakia space.
\begin{enumerate}
\item $V \subseteq Y_L$ is open in $(Y_L,\tau)$ iff $V = U\cap Y_L$ for some clopen upset $U$ of $X_L$.
\item $F \subseteq Y_L$ is closed in $(Y_L,\tau)$ iff $F = D \cap Y_L$ for some clopen downset $D$ of $X_L$.
\item The restriction to $Y_L$ of the order $\le$ on $X_L$ is the specialization order of $(Y_L,\tau)$. Thus, $\cl_\tau(\{y\}) = \down y \cap Y_L$ for each $y \in Y_L$.
\end{enumerate}
\end{lemma}

\begin{proof}
(1) The right-to-left implication follows from the definition of $\tau$. For the left-to-right implication, let $V$ be open in $(Y_L,\tau)$. By definition, $V = U \cap Y_L$ for some open upset $U$ of $X_L$. By Lemma~\ref{lem: basic facts}(3), $U$ is a union of clopen upsets. Therefore, there is $S \subseteq L$ such that $V = \bigcup \{ \varphi(s) \cap Y_L \mid s \in S\}$. Let $a = \bigvee S$.
Then $V \subseteq \varphi(a) \cap Y_L$. To see the reverse inclusion, let $y \in \varphi(a) \cap Y_L$. Since $y \in Y_L$,
we have $\down y$ is clopen. Therefore, by \cite[Lem.~5.1]{BGJ16}, $y$ is a completely prime filter of $L$. Since $a \in y$,
there is $s \in S$ with $s \in y$. Thus, $y \in \varphi(s) \cap Y_L \subseteq V$. This shows $V = \varphi(a) \cap Y_L$,
completing the proof of (1).

(2) This follows from (1).

(3) Let $\le_\tau$ be the specialization order of $(Y_L,\tau)$. Suppose $x, y \in Y_L$. If $x \le y$, then for every open upset $V$ of $X_L$,
from $x \in V$ it follows that $y \in V$. Therefore, by the definition of $\tau$,
if $U$ is an open set in $(Y_L,\tau)$, then $x \in U$ implies $y \in U$. Thus, $x \le_\tau y$. Conversely,
if $x \not\le y$, then there is a clopen upset $V$ of $X_L$ with $x \in V$ and $y \notin V$. Therefore, $V \cap Y_L$ is open in $(Y_L,\tau)$
containing $x$ and missing $y$.
Thus, $x \not\le_\tau y$, and hence $\le_\tau$ is the restriction of $\le$ to $Y_L$.
\end{proof}

\begin{proposition} \label{prop: YL = pt(L)}
The space $(Y_L, \tau)$ is $($homeomorphic to$)$ the space $\pt(L)$ of points of $L$.
\end{proposition}

\begin{proof}
If we view points of $L$ as completely prime filters, then $Y_L=\pt(L)$ by \cite[Lem.~5.1]{BGJ16}.
The open sets of $\pt(L)$ are the sets of the form $\{ y \in \pt(L) \mid a \in y\}$
as $a$ ranges over the elements of $L$. These are exactly $\varphi(a) \cap Y_L$. Therefore, by Lemma~\ref{lem: tau}(1), these are precisely
the open sets of $(Y_L,\tau)$. Thus, $(Y_L, \tau)$ is (homeomorphic to) the space $\pt(L)$ of points of $L$.
\end{proof}

We next characterize when $L$ is spatial (see also \cite[Sec.~2.11]{PS00}).

\begin{theorem} \label{prop: dense = spatial}
A frame $L$ is spatial iff $Y_L$ is dense in $(X_L,\pi)$.
\end{theorem}

\begin{proof}
Suppose that $L$ is spatial. Let $U$ be a nonempty open subset of $(X_L, \pi)$. Then there are $a,b \in L$ with
$\varnothing \ne \varphi(a) \setminus \varphi(b) \subseteq U$. Therefore, $a \not\le b$. Since $L$ is spatial, there is $y \in Y_L$ with
$a \in y$ and $b \notin y$. Then $y \in \varphi(a) \setminus \varphi(b)$, so $Y_L \cap U \ne \varnothing$. Thus,
$Y_L$ is dense in $(X_L,\pi)$.

Conversely, suppose that $Y_L$ is dense in $(X_L,\pi)$. Let $a, b \in L$ with
$a \not\le b$. Then $\varphi(a) \setminus \varphi(b)$ is a nonempty open subset of $X_L$. Since $Y_L$ is dense in
$(X_L,\pi)$, there is $y\in Y_L \cap (\varphi(a) \setminus \varphi(b))$.
Therefore, $a \in y$ and $b \notin y$. Thus, $L$ is spatial.
\end{proof}

We now turn to the question of spatiality of $N(L)$. For a topological space $S$, we let $\F(S)$ be the closed sets of $S$.
If we need to specify the topology $\pi$, we write $\F_\pi(S)$.

\begin{lemma} \label{lem: gamma}
For a frame $L$, the map $\gamma : N(X_L) \to \F_\pi(Y_L)$, defined by $\gamma(N) = N \cap Y_L$, is an onto coframe homomorphism.
\end{lemma}

\begin{proof}
Clearly $\gamma$ preserves finite joins since finite joins in both $N(X_L)$ and $\F_\pi(Y_L)$ are finite unions.
To see that $\gamma$ preserves arbitrary meets, let $\{N_\alpha\}$ be a family of nuclear sets of $X_L$. If $N$ is their meet,
then $N \subseteq \bigcap N_\alpha$, so
\[
\gamma(N) = N \cap Y_L \subseteq \left(\bigcap N_\alpha \right) \cap Y_L = \bigcap (N_\alpha \cap Y_L) = \bigcap \gamma(N_\alpha).
\]
For the reverse inclusion, let
$y \in \bigcap \gamma(N_\alpha)$. Then $y \in \bigcap (N_\alpha \cap Y_L)$. So $\{y\} \in N(X_L)$ and $y \in N_\alpha$ for each $\alpha$.
Therefore, $y \in \bigwedge N_\alpha$ by Theorem~\ref{prop: meet in NX}. Thus, $\gamma(\bigwedge N_\alpha) = \bigcap \gamma(N_\alpha)$,
and hence $\gamma$ is a coframe homomorphism.
To see that $\gamma$ is onto, let $F\in \F_\pi(Y_L)$.
Then $F = D \cap Y_L$ for some $D\in \F_\pi(X_L)$.
Since $(X_L,\pi)$ is a Stone space, $D$ is the intersection of clopen sets $U$ containing it.
Because clopen sets are nuclear (see Lemma~\ref{lem: basic nuclear facts}(1)), we have
\[
F = D \cap Y_L = \bigcap \{ U \cap Y_L \mid D \subseteq U \} = \bigcap \{ \gamma(U) \mid D \subseteq U \}
= \gamma\left(\bigwedge \{U \mid D\subseteq U\} \right).
\]
Thus, $\gamma$ is onto.
\end{proof}

We will prove that $N(L)$ is spatial iff $\gamma$ is 1-1. For this we require the following lemma.

\begin{lemma} \label{lem: join prime}
The join-prime elements of $N(X_L)$ are precisely the singletons $\{y\}$ with $y \in Y_L$.
\end{lemma}

\begin{proof}
Since binary join in $N(X_L)$ is union, it is clear that $\{y\}$ is join-prime in $N(X_L)$ for each $y\in Y_L$.
Conversely, suppose that $N$ is join-prime in $N(X_L)$. If there are $x \ne y$ in $N$, then there is a clopen set $U$ containing
$x$ but not $y$. Since $U \cup (X_L\setminus U)=X_L$, we see that $N \subseteq U \cup (X_L \setminus U)$. But $N$ is not contained in
either. This is a contradiction to $N$ being join-prime in $N(X_L)$ since $U$ and $X_L \setminus U$ are clopen, hence belong to
$N(X_L)$. Thus, $N$ is a singleton $\{y\}$ with $y\in Y_L$.
\end{proof}

Lemma~\ref{lem: join prime} has several useful consequences. First we show that there is a bijection between the points of $L$ and
the points of $N(L)$.


\begin{proposition} \label{prop: YL = pt(NL)}
There is a bijection between $\pt(L)$ and $\pt(N(L))$.
\end{proposition}

\begin{proof}
Let $Z_L$ be the set of join-prime elements of $N(X_L)$. By Lemma~\ref{lem: join prime}, $Z_L$ is in 1-1 correspondence with $Y_L$.
Since $N(L)$ is dually isomorphic to $N(X_L)$, the join-prime elements of $N(X_L)$ are in 1-1 correspondence with the meet-prime
elements of $N(L)$. But the meet-prime elements of $N(L)$ are in 1-1 correspondence with the points of $N(L)$. Therefore, there is
a 1-1 correspondence between $Y_L$ and $\pt(N(L))$. Thus, by Proposition~\ref{prop: YL = pt(L)}, there is a 1-1 correspondence between $\pt(L)$ and $\pt(N(L))$.
%
\end{proof}



We next derive the following well-known theorem.

\begin{theorem} \cite[Cor.~3.5]{NR87}
If $N(L)$ is spatial, then $L$ is spatial.
\end{theorem}

\begin{proof}
Suppose that $N(L)$ is spatial. Since points are in 1-1 correspondence with meet-prime elements, each element of $N(L)$ is a meet of
meet-prime elements of $N(L)$. As $N(L)$ is dually isomorphic to $N(X_L)$, every element of $N(X_L)$ is a join of join-prime elements
of $N(X_L)$. To see that $L$ is spatial, let $a, b \in L$ with $a \not\le b$. Then $\varphi(a) \not\subsetneq \varphi(b)$. Since these
sets are clopen, they are in $N(X_L)$. Therefore, there is a join-prime element $P$ of $N(X_L)$ such that $P\subseteq \varphi(a)$ and
$P \not\subseteq \varphi(b)$. By Lemma~\ref{lem: join prime}, there is $y \in Y_L$ with $y \in \varphi(a)$ and $y \notin \varphi(b)$.
Thus, $a \in y$ and $b \notin y$. Since $y \in \pt(L)$, we conclude that $L$ is spatial.
\end{proof}

Finally, we use Lemmas~\ref{lem: gamma} and~\ref{lem: join prime} to obtain the following characterization of spatiality of $N(L)$.

\begin{theorem} \label{thm: spatial}
For a frame $L$, the following conditions are equivalent.
\begin{enumerate}
\item The frame $N(L)$ is spatial.
\item If $N \in N(X_L)$ is nonempty, then so is $N \cap Y_L$.
\item $\gamma:N(X_L)\to\F_\pi(Y_L)$ is 1-1.
\item $\gamma:N(X_L)\to\F_\pi(Y_L)$ is an isomorphism.
\item $N(L)$ is isomorphic to $\Op_\pi(Y_L)$.
\end{enumerate}
\end{theorem}

\begin{proof}
(1) $\Rightarrow$ (2). Suppose $N \ne \varnothing$. Since $N(L)$ is spatial, each element of $N(L)$ is a meet of meet-prime elements,
and hence each element of $N(X_L)$ is a join of join-prime elements. Therefore, $N \ne \varnothing$ implies there is a join-prime
$P\in N(X_L)$ with $P \subseteq N$. But then $N \cap Y_L \ne \varnothing$ by Lemma~\ref{lem: join prime}.

(2) $\Rightarrow$ (3). It is sufficient to show that for $N,M \in N(X_L)$ with $N \not\subseteq M$, there is $y \in Y_L$ with
$y \in N \setminus M$. Suppose that $N \not\subseteq M$. Since $X_L$ is a Stone space and $M$ is closed, there is
a clopen set $U$ with $M \subseteq U$ and $N \not\subseteq U$. Therefore, $N \setminus U$ is a nonempty nuclear set by
Lemma~\ref{lem: basic nuclear facts}(3). By (2), there is $y \in Y_L$ with $y \in N \setminus U$. Thus, $y \in N \setminus M$.

(3) $\Rightarrow$ (4). By Lemma~\ref{lem: gamma}, $\gamma:N(X_L)\to\F_\pi(Y_L)$ is an onto coframe homomorphism. By (3), $\gamma$
is 1-1. Thus, $\gamma$ is an isomorphism.

(4) $\Rightarrow$ (5). This is obvious since $N(L)$ is dually isomorphic to $N(X_L)$ and $\Op_\pi(Y_L)$ is dually isomorphic to
$\F_\pi(Y_L)$.

(5) $\Rightarrow$ (1). This is clear.
\end{proof}

\begin{remark}
When both $L$ and $N(L)$ are spatial, then it is a consequence of Proposition~\ref{prop: YL = pt(L)} and Theorem~\ref{thm: spatial} that
$L$ is isomorphic to the frame of opens of $(Y_L,\tau)$, while $N(L)$ is isomorphic to the frame of opens of $(Y_L,\pi)$.
\end{remark}

We conclude this section by deriving the characterization of when $N(L)$ is spatial given in \cite{NR87} in terms of essential primes.
Let $L$ be a frame and $a \in L$. We recall that a meet-prime element $p\ge a$ is a \emph{minimal prime} (with respect to $a$) if
$p$ is minimal among meet-primes $q\ge a$. Let $\func{Min}(a)$ be the set of all minimal primes (with respect to $a$). If
$a = \bigwedge \func{Min}(a)$, then $p\in \func{Min}(a)$ is an \emph{essential prime} provided
$a \ne \bigwedge \{ q \in \func{Min}(a) \mid p \ne q\}$.

We next characterize minimal primes and essential primes in terms of $X_L$ and $Y_L$. For this we require the following lemma.
\begin{lemma} \label{lem: meet prime}
Let $L$ be a frame, $X_L$ its Esakia space, $a \in L$, and $S \subseteq L$.
\begin{enumerate}
\item $a$ is meet prime iff $\varphi(a)=X_L\setminus\down y$ for some $y\in Y_L$.
\item $a = \bigwedge S$ iff $\varphi(a) = X_L \setminus \down \left( X_L \setminus \Int_\pi\left(\bigcap\{ \varphi(s) \mid s\in S\}\right) \right)$.
\item $a = \bigwedge S$ iff $X_L\setminus\varphi(a) = \down \cl_\pi \bigcup \{ X_L\setminus\varphi(s) \mid s\in S\}$.
\end{enumerate}
\end{lemma}

\begin{proof}
(1). A characterization of join primes of $L$ is given in \cite[Thm.~2.7(1)]{BB08}. Dualizing the proof yields the result.

(2). See \cite[Lem.~2.3(3)]{BB08}.

(3). This is immediate from (2).
\end{proof}

In order to distinguish between minimal primes and essential primes of $L$, for a clopen downset $D$ of $X_L$, we will look at $\max(D) \cap Y_L$ and $\max(D \cap Y_L)$. It is clear that $\max(D) \cap Y_L \subseteq \max(D \cap Y_L)$. The reverse inclusion does not always hold, as the following example shows.

\begin{example}
Let $\beta(\mathbb{N})$ be the Stone-\v{C}ech compactification of the discrete space $\mathbb{N}$, and let $X$ be the disjoint union of $\beta (\mathbb{N})$ and the two-element discrete space $\{x,y\}$. Let $\mathbb{N}^* = \beta (\mathbb{N}) \setminus \mathbb{N}$ and define $\le$ to be the least partial order on $X$ satisfying $z \le y$ for each $z \in X$ and $x \le z$ iff $z = x$ or $z \in \mathbb{N}^*$.
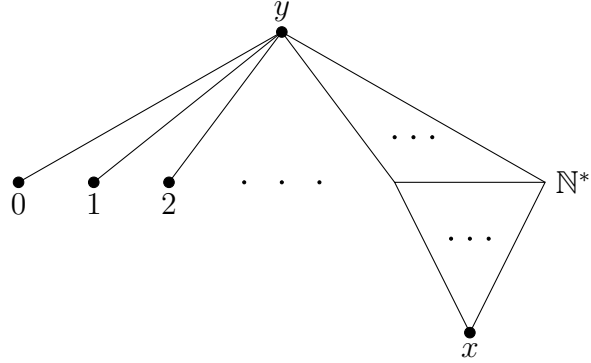
\begin{center}
\begin{tikzpicture}
\draw [fill] (0,0) circle[radius=.07];
\draw [fill] (1,0) circle[radius=.07];
\draw [fill] (2,0) circle[radius=.07];
\node [below] at (0,0) {0};
\node [below] at (1,0) {1};
\node [below] at (2,0) {2};
\node [above] at (3.5, 2) {$y$};
\node [below] at (6,-2) {$x$};
\node [right] at (7,0) {$\mathbb{N}^*$};
\draw [fill] (3,0) circle[radius=.02];
\draw [fill] (3.5,0) circle[radius=.02];
\draw [fill] (4,0) circle[radius=.02];
\draw (6,-2) -- (5,0) -- (7,0) -- (6,-2);
\draw [fill] (3.5, 2) circle[radius=.07];
\draw [fill] (6,-2) circle[radius=.07];
\draw (0,0) -- (3.5,2);
\draw (1,0) -- (3.5, 2);
\draw (2,0) -- (3.5, 2);
\draw (5,0) -- (3.5, 2);
\draw (7,0) -- (3.5, 2);
\draw [fill] (5.0, .6) circle[radius=.02];
\draw [fill] (5.25,.6) circle[radius=.02];
\draw [fill] (5.5,.6) circle[radius=.02];

\draw [fill] (5.75, -.75) circle[radius=.02];
\draw [fill] (6,-.75) circle[radius=.02];
\draw [fill] (6.25,-.75) circle[radius=.02];
\end{tikzpicture}
\end{center}
Since $\beta(\mathbb{N})$ is extremally disconnected, it is clear that so is $X$. It is also easy to see that $X$ is an Esakia space. Therefore, $X$ is an extremally order-disconnected Esakia space. Thus, the clopen upsets of $X$ form a frame $L$, and we identify $X$ with $X_L$. Then $Y_L$ is identified with $\mathbb{N} \cup \{x,y\}$.

Let $D = X \setminus \{y\}$. Then $D$ is a clopen downset, $\max(D) = \beta(\mathbb{N})$, $\max(D) \cap Y_L = \mathbb{N}$, and $\max(D \cap Y_L) = \mathbb{N} \cup \{x\}$. Therefore, $\max(D) \cap Y_L$ is a proper subset of $\max(D \cap Y_L)$.
\end{example}

\begin{proposition} \label{prop: essential}
Let $L$ be a frame, $X_L$ its Esakia space, $a \in L$, and $p\ge a$ a meet-prime element of $L$.
\begin{enumerate}
\item $p$ is a minimal prime iff $\varphi(p)=X_L\setminus\down y$ for some $y \in \max[(X_L \setminus \varphi(a)) \cap Y_L]$.
\item $p$ is an essential prime iff $\varphi(p)=X_L\setminus\down y$ for some $y \in \max(X_L \setminus \varphi(a)) \cap Y_L$.
\end{enumerate}
\end{proposition}

\begin{proof}
(1). First suppose that $\varphi(p)=X_L\setminus\down y$ for some $y \in \max[(X_L \setminus \varphi(a)) \cap Y_L]$.
Let $q\ge a$ be a meet-prime with $q\le p$. Then $\varphi(q) = X_L\setminus \down z$ for some $z \in Y_L$ by Lemma~\ref{lem: meet prime}(1). Since $a \le q$, we have
$z \in X_L \setminus \varphi(a)$, so $z \in (X_L \setminus \varphi(a)) \cap Y_L$. From $q \le p$ it follows that $\down y \subseteq \down z$,
so $y \le z$. Since $y \in \max[(X_L\setminus U) \cap Y_L]$, we conclude that $y = z$, and hence $q = p$. Thus, $p$ is a minimal prime.

Conversely, suppose $p$ is a minimal prime. Then $\varphi(p)=X_L\setminus\down y$ for some $y \in (X_L\setminus \varphi(a)) \cap Y_L$.
Let $y \le z$ with $z \in (X_L \setminus \varphi(a)) \cap Y_L$. There is a meet-prime $q\in L$ such that $\varphi(q) = X_L \setminus \down z$
and $\varphi(a) \subseteq \varphi(q) \subseteq \varphi(p)$. Therefore, $a\le q\le p$, and since $p$ is a minimal prime, $q = p$. Thus, $z = y$,
proving that $y \in \max[(X_L \setminus \varphi(a)) \cap Y_L]$.

(2). Suppose $a=\bigwedge \func{Min}(a)$. By Lemma~\ref{lem: meet prime}(3),
\[
X_L\setminus\varphi(a) = \down \cl_\pi \bigcup \{ X_L\setminus\varphi(p) \mid p\in \func{Min}(a)\}.
\]
By (1),
\begin{align*}\label{equation}
X_L\setminus\varphi(a) &= \down \cl_\pi \bigcup \{ \down y \mid y \in \max[(X_L \setminus \varphi(a)) \cap Y_L] \} \\
&= \down \cl_\pi \down \max[(X_L \setminus \varphi(a)) \cap Y_L] \tag{$\dagger$}\\
&= \down \cl_\pi \max[(X_L \setminus \varphi(a)) \cap Y_L].
\end{align*}

Now suppose $\varphi(p)=X_L\setminus \down y$ for some $y \in \max(X_L \setminus \varphi(a)) \cap Y_L$. Set
\[
T = \max[(X_L \setminus \varphi(a)) \cap Y_L] \setminus \{y\}.
\]
If $\varphi(a) = X_L \setminus \down \cl_\pi(T)$, then
$X_L \setminus \down \cl_\pi(T) \subseteq X \setminus \down y$, so $\down y \subseteq \down \cl_\pi(T)$. Since
$y \in \max(X_L \setminus \varphi(a))$ and $\cl_\pi(T) \subseteq X_L \setminus \varphi(a)$, it follows that $y \in \cl_\pi(T)$.
But $\down y$ is a clopen set containing $y$ and $\down y \cap T = \varnothing$ by definition of $T$. Thus, $y \notin \cl_\pi(T)$.
The obtained contradiction proves that $X_L \setminus \varphi(a) \ne \down \cl_\pi(T)$, so $p$ is an essential prime.

Conversely, suppose $p$ is an essential prime. We have $\varphi(p)=X_L \setminus \down y$ for some $y \in \max[(X_L \setminus \varphi(a)) \cap Y_L]$. As above set $T = \max[(X_L \setminus \varphi(a)) \cap Y_L] \setminus \{y\}$.  If $X_L \setminus \varphi(a) = \down \cl_\pi(T)$, then $a = \bigwedge (\func{Min}(a) \setminus \{p\})$ by Lemma~\ref{lem: meet prime}(3), which is false since $p$ is an essential prime. Therefore, $X_L \setminus \varphi(a) \not\subseteq \down \cl_\pi(S)$. Thus,
$\max(X_L \setminus \varphi(a)) \not\subseteq \cl_\pi(S)$. Let $x \in \max(X_L \setminus \varphi(a))$ with $x \notin \cl_\pi(S)$. Then there is a clopen neighborhood $U$ of $x$ with $U \cap S = \varnothing$. Let $V$ be a clopen upset containing $x$. Since $x \in \max(X_L \setminus \varphi(a))$, by (\ref{equation}),
$x \in \cl_\pi(\max[(X_L \setminus \varphi(a)) \cap Y_L])$. The set $U\cap V$ is an open neighborhood of $x$ contained in $U$, so
\[
U \cap V \cap \max[(X_L \setminus \varphi(a)) \cap Y_L] \ne \varnothing \mbox{ and } U \cap V \cap S = \varnothing.
\]
Therefore, $y \in V$. Since this is true for all clopen upsets containing $x$, it follows that $x \le y$. Thus, $x = y$ as
$x \in \max(X_L \setminus \varphi(a))$. Consequently, $y \in \max(X_L \setminus \varphi(a)) \cap Y_L$.
\end{proof}

We are ready to give an alternate proof of the results of \cite{NR87}.

\begin{theorem}
Let $L$ be a frame.
\begin{enumerate}
\item \cite[p.~264]{NR87} $L$ is spatial iff $a = \bigwedge \func{Min}(a)$ for each $a\in L$.
\item \cite[Thm.~3.4]{NR87} $N(L)$ is spatial iff each $1 \ne a \in L$ has an essential prime.
\end{enumerate}
\end{theorem}

\begin{proof}
(1) First suppose that $L$ is spatial. By Theorem~\ref{prop: dense = spatial}, $Y_L$ is dense in $(X_L,\pi)$. Let $a\in L$. To see that $a = \bigwedge \func{Min}(a)$, we need to show that $X_L \setminus \varphi(a) = \down \cl_\pi \max[(X_L\setminus \varphi(a))\cap Y_L]$. For this it is sufficient to show that
$\max(X_L\setminus \varphi(a))\subseteq\cl_\pi(\max[(X_L\setminus \varphi(a))\cap Y_L])$. Let
$x\in \max(X_L\setminus \varphi(a))$ and let $U$ be a clopen neighborhood of $x$. Then $U\setminus\varphi(a)\ne\varnothing$.
Since $Y_L$ is dense, $(U\setminus\varphi(a))\cap Y_L\ne\varnothing$. Thus, $x\in \cl_\pi(\max[(X_L\setminus \varphi(a))\cap Y_L])$.

Conversely, suppose that $a = \bigwedge \func{Min}(a)$ for each $a\in L$. By Theorem~\ref{prop: dense = spatial}, it is
sufficient to show that $Y_L$ is dense in $(X_L,\pi)$. Let $U$ be nonempty clopen. Then $\down U$ is a clopen
downset, so there is $a\in L$ with $\varphi(a)=X_L\setminus \down U$. Since $a = \bigwedge \func{Min}(a)$, by Lemma~\ref{lem: meet prime}(3) and Proposition~\ref{prop: essential}(1), $\down U = \down \cl_\pi \max(\down U \cap Y_L)$. Because $\max(\down U) = \max(U)$, we have $\max(U) \subseteq \cl_\pi \max(\down U \cap Y_L)$. Therefore, $U \cap \cl_\pi \max(\down U \cap Y_L) \ne \varnothing$. Since $U$ is clopen, $U \cap \max(\down U \cap Y_L) \ne \varnothing$, so $U \cap Y_L \ne \varnothing$. Thus, $Y_L$ is dense in $(X_L,\pi)$.

(2) First suppose that $N(L)$ is spatial. Let $a \ne 1$. Then
$\max(X_L\setminus \varphi(a))$ is a nonempty nuclear set by Theorem~\ref{thm: w_a}(3).
Therefore, $\max(X_L\setminus \varphi(a)) \cap Y_L \ne \varnothing$ by Theorem~\ref{thm: spatial}. If $y \in \max(X_L \setminus U) \cap Y_L$,
then the $p$ such that $\varphi(p)=X_L \setminus \down y$ is an essential prime by Proposition~\ref{prop: essential}(2).

Conversely, suppose that each $1 \ne a \in L$ has an essential prime. Let $N$ be a nonempty nuclear set.
Then $\down N$ is a nonempty clopen downset, so there is $a\ne 1$ such that $\varphi(a)=X_L\setminus \down N$.
Therefore, $a$ has an essential prime $p$. By Proposition~\ref{prop: essential}(2), there is $y \in \max(X_L\setminus \varphi(a)) \cap Y_L$ with
$\varphi(p) = X_L \setminus \down y$, so $\max(X_L\setminus \varphi(a)) \cap Y_L\ne\varnothing$. Thus,
$\max(\down N) \cap Y_L = \max(N) \cap Y_L$ is nonempty, and so $N\cap Y_L$ is nonempty. Consequently, $N(L)$ is spatial by
Theorem~\ref{thm: spatial}.
\end{proof}

\section{When is $N(L)$ boolean?} \label{sec: boolean}

In \cite{BM79} Beazer and Macnab proved that if $L$ is boolean, then $N(L) \cong L$. {To see this in terms of $X_L$,} if $L$ is boolean, 
then the order of $X_L$ is equality, and $L$ is isomorphic to the clopens of $X_L$. Since the order of $X_L$ is equality, nuclear sets of $X_L$ are 
exactly clopen sets of $X_L$. Thus, $N(L)$ is isomorphic to $L$, yielding \cite[Cor.~1]{BM79}. 

Beazer and Macnab also gave a characterization of when $N(L)$ is boolean \cite[Thm.~2]{BM79}.  
Let $L$ be a frame. Recall that $d \in L$ is \emph{dense} if $\lnot d = 0$. If $a \in L$, then $\up a$ is a frame, and $d \ge a$
is dense in $\up a$ iff $d \to a = a$. Following \cite[Sec.~1.4]{Ple00} and \cite[Sec.~6]{BGJ16}, we call $L$ \emph{scattered}
if for each $a \in L$ the principal upset $\up a$ has a smallest dense element. Using this definition, \cite[Thm.~2]{BM79} can
be phrased as $N(L)$ is boolean iff $L$ is scattered.
In this section we give a characterization of when $N(L)$ is boolean in terms of $X_L$, from which we derive 
\cite[Thm.~2]{BM79}.

For a topological space $S$, let $\RC(S)$ be the boolean frame of regular closed sets of $S$. By \cite[Prop.~4.12]{BGJ13},
if $L$ is a frame, then its booleanization $B(N(L))$ is dually isomorphic to $\RC(X_L)$.

\begin{theorem} \label{thm: boolean}
Let $L$ be a frame and $X_L$ its Esakia space. Then the following conditions are equivalent.
\begin{enumerate}
\item $N(L)$ is boolean;
\item $N(X_L) = \RC(X_L)$;
\item $\max(D)$ is clopen for each clopen downset $D$ of $X_L$.
\item $L$ is scattered.
\end{enumerate}
\end{theorem}

\begin{proof}
(1) $\Leftrightarrow$ (2). See \cite[Thm.~4.14]{BGJ13}.

(2) $\Rightarrow$ (3). Let $D$ be a clopen downset. Then $D$ is nuclear by Lemma~\ref{lem: basic nuclear facts}(1),
so $\max(D)$ is nuclear by Corollary~\ref{cor: max}(1). Let $V = D \setminus \max(D)$. Then $V$ is open, so $\cl_\pi(V)$ is regular
closed, and hence nuclear by Lemma~\ref{lem: basic nuclear facts}(2). Clearly $V$ is a downset. We show that $\cl_\pi(V)$ is a downset. 
Let $x \le y \in \cl_\pi(V)$. Since
$\cl_\pi(V) \subseteq D$, we have $x\in D$. If $x\notin\max(D)$, then $x\in V\subseteq\cl_\pi(V)$. If $x\in\max(D)$, then $x=y$, so
again $x \in \cl_\pi(V)$. Thus, $\cl_\pi(V)$ is a downset. Since $\cl_\pi(V)$ is nuclear, it is then a clopen downset, so
$F := \cl_\pi(V) \cap \max(D)$ is nuclear by Lemma~\ref{lem: basic nuclear facts}(3), and hence regular closed by (2).
This implies that $F = \cl_\pi(V)\setminus V$. Since $\Int_\pi(\cl_\pi(V) \setminus V) = \varnothing$ and $F$ is regular closed,
this forces $F = \varnothing$. Thus, $\cl_\pi(V) = V$, so $V$ is clopen, and hence $\max(D) = D \setminus V$ is clopen.

(3) $\Rightarrow$ (2). By Lemma~\ref{lem: basic nuclear facts}(2), it suffices to prove that each nuclear subset of $X_L$ is regular 
closed. Let $N \in N(X_L)$. To show that $N$ is regular closed, it is enough to prove that $N \subseteq \cl_\pi (\Int_\pi (N))$. Let 
$x \in N$. For each clopen $U$ containing $x$ we have $U \cap N$ is a nonempty nuclear set, so $\down (U\cap N)$ is clopen. Thus, 
$\max\down(U\cap N)$ is clopen. But $\max\down(U\cap N)=\max(U \cap N)$, so $\max(U \cap N)$ is clopen, and hence it is contained 
in $\Int_\pi(N)$. Therefore, $U \cap \Int_\pi(N) \ne \varnothing$, which proves $x \in \cl_\pi(\Int_\pi(N))$. Thus, $N$ is regular 
closed.

(3) $\Leftrightarrow$ (4). See \cite[Thm.~6.6]{BGJ16}.
\end{proof}

\section{Soberification and the theorems of Simmons and Isbell} \label{sec: Simmons}

Let $S$ be a topological space. In this section we show how to derive the results of Simmons \cite{Sim80} and Isbell \cite{Isb91} relating
$S$ being weakly scattered and scattered to that of $N(\Op S)$ being spatial and boolean. For this we recall that the \emph{soberification}
of $S$ is the space $\pt(\Op S)$.
Viewing points of $\Op S$ as completely prime filters, we have the mapping $\varepsilon : S \to \pt(\Op S)$ sending $s$ to the
completely prime filter $\{ U \in \Op S \mid s \in U\}$. It is well known (see, e.g., \cite[Sec.~II.1]{Joh82}) that $\varepsilon$
is continuous, is an embedding iff $S$ is $T_0$, and induces an isomorphism of the frames of open sets.

\begin{proposition} \label{prop: sober}
For a topological space $S$, the space $(Y_{\Op S},\tau)$ is $($homeomorphic to$)$ the soberification of $S$.
\end{proposition}

\begin{proof}
By Proposition~\ref{prop: YL = pt(L)}, if $L$ is a frame, then $(Y_L,\tau)$ is (homeomorphic to) $\pt(L)$. Thus, $(Y_{\Op S},\tau)$
is (homeomorphic to) $\pt(\Op S)$, and hence is (homeomorphic to) the soberification of $S$.
\end{proof}

Let $S$ be a topological space and $T$ a subspace of $S$. We recall from the introduction that a point $x \in T$ is
\emph{weakly isolated} in $T$ if there is an open subset $U$ of $S$ such that $x \in T \cap U \subseteq \overline{\{x\}}$;
and that $S$ is \emph{weakly scattered} if each nonempty closed subspace has a weakly isolated point.

\begin{lemma} \label{lem: weakly isolated implies maximal}
Let $L$ be a spatial frame and $D$ a clopen downset of $(X_L,\pi)$. If $y \in D\cap Y_L$ is weakly isolated in $D\cap Y_L$ viewed as a
subspace of $(Y_L,\tau)$, then $y \in \max(D)$.
\end{lemma}

\begin{proof}
If $y \in D \cap Y_L$ is weakly isolated, then there is an open subset $U$ of $(Y_L,\tau)$ such that
$y \in D\cap Y_L \cap U \subseteq \cl_\tau(\{y\})=\down y \cap Y_L$, where the equality follows from Lemma~\ref{lem: tau}(3). By Lemma~\ref{lem: tau}(1), we may write $U = V \cap Y_L$
for some clopen upset $V$ of $X_L$. Therefore, $D \cap V \cap Y_L \subseteq \down y$. This implies
$D \cap V \cap (X \setminus \down y) \cap Y_L = \varnothing$, so $D \cap V \cap (X \setminus \down y) = \varnothing$ as $Y_L$ is
dense in $(X_L,\pi)$ by Theorem~\ref{prop: dense = spatial}. Thus, $D \cap V \subseteq \down y$. If $y \le z$ for some
$z \in D$, then $z \in D \cap V$ since $V$ is an upset. Therefore, $z \in \down y$, which forces $z = y$. Thus, $y \in \max(D)$.
\end{proof}

\begin{theorem} \label{thm: weakly scattered}
Let $L$ be a spatial frame. Then $N(L)$ is spatial iff $(Y_L, \tau)$ is weakly scattered.
\end{theorem}

\begin{proof}
First suppose that $N(L)$ is spatial. To show $(Y_L,\tau)$ is weakly scattered, let $F\in\F_\tau(Y_L)$ be nonempty. By Lemma~\ref{lem: tau}(2),
$F = D \cap Y_L$ for some clopen downset $D$. Since $\max(D)$ is nuclear (see
Corollary~\ref{cor: max}(1)) and nonempty (as $\down \max(D) = D$),
we have $\max(D) \cap Y_L \ne \varnothing$ by Theorem~\ref{thm: spatial}. Let $x \in \max(D) \cap Y_L$. We show that $x$ is weakly isolated
in $F$. Since $x\in Y_L$, the downset $\down x$ is clopen. Set
\[
N = \cl_\pi\left(\bigcup \{ \down y \setminus \down x \mid y \in F, y \not\le x\} \right).
\]
Since each $\down y \setminus \down x$ is clopen, hence nuclear, $N$ is nuclear by Lemma~\ref{lem: nuclear}. If $y \in F$, then
$\down y \setminus \down x \subseteq D$ since $D$ is a downset. Therefore, the union is in $D$, and so $N \subseteq D$. Moreover,
$\down N \subseteq D$ since $D$ is a downset. We see that $x \notin N$; for, $\down x$ is a clopen neighborhood of $x$ and
$\down x \cap (\down y \setminus \down x) = \varnothing$ for each $y \not\le x$ with $y \in Y_L$. Thus, $\down x$ misses the union,
and so $x \notin N$. Since $x \in \max(D) \setminus N$ and $\down N \subseteq D$, it follows that $x \notin \down N$. Hence
$X_L \setminus \down N$ is a clopen upset containing $x$. We have $x \in (X \setminus \down N) \cap F \subseteq \down x$ since if
$y \in F$ with $y \not\le x$, then $y \in N$. This shows that $x$ is weakly isolated in $F$. Consequently, $Y_L$ is weakly scattered.

Conversely, suppose that $Y_L$ is weakly scattered. By Theorem~\ref{thm: spatial}, it is sufficient to show that
if $N$ is a nonempty nuclear subset of $X_L$, then $N \cap Y_L \ne \varnothing$. Since $L$ is
spatial, $Y_L$ is dense in $(X_L,\pi)$ by Theorem~\ref{prop: dense = spatial}. Therefore, $\down N \cap Y_L \ne \varnothing$ since $\down N$
is a nonempty clopen. As $\down N \cap Y_L$ is a nonempty closed set in $(Y_L,\tau)$, since $Y_L$ is weakly scattered, there is a weakly isolated
point $y$ of $N \cap Y_L$. By Lemma~\ref{lem: weakly isolated implies maximal}, $y \in \max(\down N) = \max(N)$. Thus, $N \cap Y_L \ne \varnothing$.
\end{proof}

\begin{example}
The following example shows that the hypothesis of Theorem~\ref{thm: weakly scattered} that $L$ is spatial
is necessary. Let $L$ be the frame of regular open sets of the space $S = [0,1] \cup \{2\}$ with the usual Euclidean topology. Since $L$
is boolean, points of $L$ correspond to isolated points of $S$ (see, e.g., \cite[Sec.~II.5.4]{PP12}). Consequently, $Y_L$ is a singleton, and so $Y_L$ is weakly scattered.
As $L$ is boolean, $L = N(L)$. But $L$ is not spatial since it is not atomic. Thus, $N(L)$ is not spatial.
\end{example}

We next show how to derive Isbell's theorem \cite[Thm.~7]{Isb91} as a consequence of Theorem~\ref{thm: weakly scattered}. For this we require the following simple lemma.

\begin{lemma} \label{lem: weakly scattered implies sober}
If $S$ is a weakly scattered $T_0$-space, then $S$ is sober.
\end{lemma}

\begin{proof}
Let $F$ be a closed irreducible subset of $S$. Then $F$ is nonempty, so there is a weakly isolated point $x \in F$. Therefore, there is
an open set $U$ of $S$ such that $x \in U \cap F \subseteq \overline{\{x\}}$. Thus, $F \setminus U$ is a proper closed subset of $F$ and
$F = (F \setminus U) \cup \overline{\{x\}}$. Since $F$ is irreducible and $F \setminus U \ne F$, we must have
$F = \overline{\{x\}}$. As $S$ is $T_0$, we conclude that $S$ is sober.
\end{proof}

\begin{theorem} \label{thm: isbell}
Let $S$ be a topological space.
\begin{enumerate}
\item $N(\Op S)$ is spatial iff the soberification of $S$ is weakly scattered.
\item \cite[Thm.~7]{Isb91} If $S$ is $T_0$, then $S$ is sober and $N(\Op S)$ is spatial iff $S$ is weakly scattered.
In particular, if $S$ is sober, then $N(\Op S)$ is spatial iff $S$ is weakly scattered.
\end{enumerate}
\end{theorem}

\begin{proof}
(1). By Proposition~\ref{prop: sober}, view $(Y_{\Op S},\tau)$ as the soberification of $S$ and apply Theorem~\ref{thm: weakly scattered}.

(2). Let $S$ be a $T_0$-space.
If $S$ is sober, then $S$ is homeomorphic to $(Y_{\Op S}, \tau)$. Hence, $N(\Op S)$ spatial implies that $S$ is weakly scattered by Theorem~\ref{thm: weakly scattered}. Conversely, if $S$ is weakly scattered, then $S$ is sober by Lemma~\ref{lem: weakly isolated implies maximal}, so $S$ is homeomorphic to $(Y_{\Op S}, \tau)$. Applying Theorem~\ref{thm: weakly scattered} then yields that $N(\Op S)$ is spatial.
\end{proof}

As we pointed in the introduction, the assumption that $S$ is sober cannot be dropped from Theorem~\ref{thm: isbell}(2), as the following example shows. For this we recall the following definition.

\begin{definition}
 \cite[p.~24]{Sim80} The \emph{front topology} on a topological space $S$ is the topology $\tau_F$ generated by $\{ U\setminus V \mid U,V \in \Op S\}$.
\end{definition}

\begin{example} \label{ex: 7.6}
Let $S$ be the set of natural numbers with the usual order. We put the Alexandroff topology on $S$, so $U$ is open
in $S$ iff $U$ is an upset.
We then have the following picture, where $\boldsymbol{n}$ is the prime filter of $\Op S$ consisting of all open sets containing $n$ and $\infty$ is the prime filter $\mathcal{O}S \setminus \{\varnothing \}$.
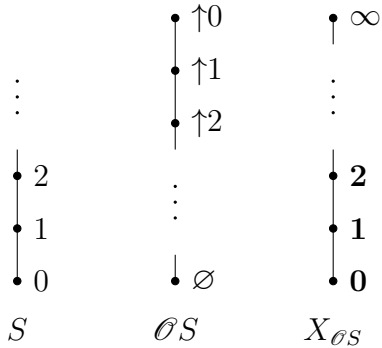
\begin{center}
\begin{tikzpicture}[scale=0.7]
\draw [fill] (0,0) circle[radius=.07];
\draw [fill] (0,1) circle[radius=.07];
\draw [fill] (0,2) circle[radius=.07];
\draw [fill] (0,3.2) circle[radius=.02];
\draw [fill] (0,3.5) circle[radius=.02];
\draw [fill] (0,3.8) circle[radius=.02];
\node [right] at (0.1,0) {0};
\node [right] at (0.1,1) {1};
\node [right] at (0.1,2) {2};
\draw (0,0) -- (0,2.5);
\node [below] at (0.0, -.5) {$S$};

\draw [fill] (3,5) circle[radius=.07];
\draw [fill] (3,4) circle[radius=.07];
\draw [fill] (3,3) circle[radius=.07];
\draw [fill] (3,1.8) circle[radius=.02];
\draw [fill] (3,1.5) circle[radius=.02];
\draw [fill] (3,1.2) circle[radius=.02];
\draw [fill] (3,0) circle[radius=.07];
\node [right] at (3.1,5) {${\uparrow}0$};
\node [right] at (3.1,4) {${\uparrow}1$};
\node [right] at (3.1,3) {${\uparrow}2$};
\node [right] at (3.1,0) {$\varnothing$};
\draw (3,5) -- (3,2.5);
\draw (3,0) -- (3, .5);
\node [below] at (3.0,-.5) {$\Op S$};

\draw [fill] (6,0) circle[radius=.07];
\draw [fill] (6,1) circle[radius=.07];
\draw [fill] (6,2) circle[radius=.07];
\draw [fill] (6,3.2) circle[radius=.02];
\draw [fill] (6,3.5) circle[radius=.02];
\draw [fill] (6,3.8) circle[radius=.02];
\draw [fill] (6,5) circle[radius=.07];
\node [right] at (6.1,0) {\textbf{0}};
\node [right] at (6.1,1) {\textbf{1}};
\node [right] at (6.1,2) {\textbf{2}};
\node [right] at (6.1,5) {$\infty$};
\draw (6, 0) -- (6, 2.5);
\draw (6,4.5) -- (6,5);
\node [below] at (6.0,-.5) {$X_{\Op S}$};
\end{tikzpicture}
\end{center}

The space
$(X_{\Op S}, \pi)$ is the one-point compactification of $(S, \tau_F)$, which is a discrete space, and the order on $X_{\Op S}$ is described in the picture. Since the downset of each $x\in X_{\Op S}$ is clopen, we see that $Y_{\Op S}=X_{\Op S}$ and $N(X_{\Op S})=\F_\pi(Y_{\Op S})$,
yielding that $N(\Op S)$ is spatial. On the other hand,
$S$ is not weakly scattered since there are no weakly isolated points in $S$. Indeed, if there was a weakly isolated
point $s\in S$, then there would exist an open set $U$ with $s \in U \cap S = U \subseteq
{\downarrow}s$. But ${\downarrow}s$ is finite for each $s \in S$, while nonempty open subsets of $S$ are infinite. This shows that
$s$ is not weakly isolated in $S$, and hence $S$ is not weakly scattered.
\end{example}

We next turn to Simmons' results. One of Simmons' main tools was to use the front topology on a topological space $S$. We denote the frame of open sets of the front topology by $\Op_F(S)$.
Since Simmons did not assume that $S$ is $T_0$, we recall
that the \emph{$T_0$-reflection} of $S$ is defined as the quotient space $S_0$ by the equivalence relation
$\sim$
given by $x \sim y$ iff $\overline{\{x\}}=\overline{\{y\}}$.
Clearly
$S_0$ is a $T_0$-space,
and the canonical map $\rho : S \to S_0$ is both an open and a closed map since each open set and hence each closed set of $S$ is saturated
with respect to $\sim$. We denote the equivalence class of $x\in S$ by $[x]$.
In the next lemma we show how the front topology is connected to the space $(X_{\Op S},\pi)$.

\begin{lemma} \label{lem: e}
The map $\varepsilon : S \to X_{\Op S}$ factors through the natural map $\rho : S \to S_0$. If $\varepsilon' : S_0 \to X_{\Op S}$ is
the induced map, then the pair $((X_{\Op S},\pi),\varepsilon')$ is a compactification of $(S_0,\tau_F)$. In particular, if $S$ is $T_0$, then $((X_{\Op S},\pi),\varepsilon)$ is a compactification of $(S,\tau_F)$.
\[
\begin{tikzcd}
S \arrow[rr, "\varepsilon"] \arrow[dr, "\rho"'] && X_{\Op S} \\
& S_0 \arrow[ru, "\varepsilon'"'] &
\end{tikzcd}
\]
\end{lemma}

\begin{proof}
Since $\varepsilon$ is continuous with respect to the open upset topology on $X_{\Op S}$, which is a $T_0$ topology,
$\varepsilon'$ factors through $\rho$.

A basic open set of $(X_{\Op S},\pi)$ has the form $\varphi(U) \setminus \varphi(V)$ for some $U,V$ open in $S$. Therefore,
$\varepsilon^{-1}(\varphi(U) \setminus \varphi(V)) = U\setminus V$. Thus, $\varepsilon$ is continuous with respect to the
front topology on $S$. Because
$(\varepsilon')^{-1}(\varphi(U) \setminus \varphi(V)) = \rho(\varepsilon^{-1}(\varphi(U)\setminus\varphi(V))$ and $\rho$ is
an open map, we see that $\varepsilon'$ is continuous with respect to the front topology on $S_0$. It is a homeomorphism onto
its image since for
$U,V \in \Op S$, we have $\varepsilon'(\rho(U)\setminus \rho(V)) = (\varphi(U) \setminus \varphi(V)) \cap \varepsilon[S]$.
To see that the image is dense, let $\varphi(U) \setminus \varphi(V)$ be nonempty.
Then $U \setminus V \ne \varnothing$, which shows that
$(\varphi(U) \setminus \varphi(V)) \cap \varepsilon[S] \ne \varnothing$. Therefore, $\varepsilon'[S_0] = \varepsilon[S]$
is dense in $(X_{\Op S},\pi)$.
Thus, $\varepsilon' : S_0 \to X_{\Op S}$ is a compactification with respect to the front topology on $S_0$.
\end{proof}

Let $A$ be a subset of $S$. Following \cite[Def.~1.2]{Sim80}, we call $x \in A$ \emph{detached} in $A$ if there is an open set $U$ of $S$
with $x \in U \cap A \subseteq [x]$. The space $S$ is \emph{dispersed} if each nonempty closed subset of $S$ has a detached point
(see \cite[Def.~1.8, Thm.~1.9]{Sim80}).

\begin{proposition} \label{prop: dispersed}
Let $S$ be a topological space and let $S_0$ be its $T_0$-reflection.
\begin{enumerate}
\item $S$ is weakly scattered iff $S_0$ is weakly scattered.
\item $S$ is dispersed iff $S_0$ is scattered.
\end{enumerate}
\end{proposition}

\begin{proof}
(1). Suppose that $S$ is weakly scattered. Let $F$ be a nonempty closed subset of $S_0$. Then $\rho^{-1}(F)$ is a nonempty closed subset of $S$,
so there is $x \in S$ and an open set $U$ of $S$ with $x \in U \cap \rho^{-1}(F) \subseteq \overline{\{x\}}$.
Therefore, $\rho(x) \in \rho(U) \cap F \subseteq \rho\left(\overline{\{x\}}\right) = \overline{\{\rho(x)\}}$. Since $\rho(U)$ is open in $S_0$,
this shows $\rho(x)$ is weakly isolated in $F$. Thus, $S_0$ is weakly scattered.

Conversely, suppose that $S_0$ is weakly scattered. Let $F$ be a nonempty closed subset of $S$. Then $\rho(F)$ is closed in $S_0$ and nonempty.
Therefore, there is $x \in S$ and an open set $V$ of $S_0$ with $\rho(x) \in V \cap \rho(F) \subseteq \overline{\{\rho(x)\}}$. Since $F = \rho^{-1}(\rho(F))$, this yields
$x \in \rho^{-1}(V) \cap F \subseteq \rho^{-1}(\overline{\{\rho(x)\}}) = \overline{\{x\}}$. Thus, $x$ is weakly isolated in $F$, and hence
$S$ is weakly scattered.

(2). First suppose that $S$ is dispersed. Let $F$ be a nonempty closed subset of $S_0$. Then $\rho^{-1}(F)$ is a nonempty closed subset of $S$.
Since $S$ is dispersed, there is a detached point $x$ of $\rho^{-1}(F)$. Therefore, there is an open set $U$ of $S$ with
$x \in U \cap \rho^{-1}(F) \subseteq [x]$. Thus, $\rho(x) \in \rho(U) \cap F \subseteq \{\rho(x)\}$, which shows $\rho(x)$ is isolated in $F$.
Consequently, $S_0$ is scattered.

Conversely, suppose that $S_0$ is scattered. Let $F$ be a nonempty closed subset of $S$. Then $\rho(F)$ is a nonempty closed subset of $S_0$.
Since $S_0$ is scattered, $\rho(F)$ has an isolated point $\rho(x)$. But then $x$ is a detached point of $F$. Therefore, $F$ has a detached point,
and hence $S$ is dispersed.
\end{proof}

For a topological space $S$, let $\F_F(S)$ denote the coframe of closed sets of the front topology on $S$.
Define $\delta : N(X_{\Op S}) \to \F_F(S)$ by $\delta(N) = \varepsilon^{-1}(N)$ for each $N \in N(X_{\Op S})$. Then $\delta$ is the
composition of $\gamma$ with the pullback map $\F_\pi(Y_{\Op S}) \to \F_F(S)$. Therefore, the following is a consequence of Lemma~\ref{lem: gamma}.

\begin{lemma} \label{lem: delta}
$\delta : N(X_{\Op S}) \to \F_F(S)$ is an onto coframe homomorphism.
\end{lemma}


\begin{theorem} \label{prop: delta}
For a topological space $S$, the following are equivalent.
\begin{enumerate}
\item $S$ is weakly scattered.
\item If $N\in N(X_{\Op S})$ is nonempty, then so is $\delta(N)$.
\item $\delta$ is 1-1.
\item $\delta$ is an isomorphism.
\end{enumerate}
\end{theorem}

\begin{proof}
(1) $\Rightarrow$ (2). Suppose $S$ is weakly scattered. Then $S_0$ is weakly scattered by Proposition~\ref{prop: dispersed}(1). Since $S_0$ is
$T_0$, we see that $S_0$ is sober by Lemma~\ref{lem: weakly scattered implies sober}. Thus, $\varepsilon : S_0 \to Y_{\Op S}$ is a homeomorphism. Therefore,
$Y_{\Op S}$ is weakly scattered, so $N(\Op S)$ is spatial by Theorem~\ref{thm: weakly scattered}. Let $\varnothing \ne N \in N(X_{\Op S})$.
Then $\gamma(N) \ne \varnothing$ by Theorem~\ref{thm: spatial}. Because $\varepsilon$ is a homeomorphism, the pullback map
$\varepsilon^{-1} : \F_\pi(Y_{\Op S}) \to \F_F(S)$ is an isomorphism. Thus, $\delta(N) = \varepsilon^{-1}(\gamma(N)) \ne \varnothing$.

(2) $\Rightarrow$ (3). Suppose $N, M\in N(X_{\Op S})$ with $N \not\subseteq M$. Since $M$ is closed, there
is a clopen set $U$ with $M \subseteq U$ and $N \not\subseteq U$. Then $N \setminus U$ is a nonempty nuclear set, so
$\delta(N \setminus U) \ne \varnothing$ by (2). Therefore, there is $s \in S$ with $\varepsilon(s) \in N \setminus U$. Thus,
$\varepsilon(s) \in N$ but $\varepsilon(s) \notin M$. Consequently, $\delta(N) = \varepsilon^{-1}(N) \not\subseteq \varepsilon^{-1}(M) = \delta(M)$,
and so $\delta$ is 1-1.

(3) $\Rightarrow$ (4). This is clear since $\delta$ is an onto coframe homomorphism by Lemma~\ref{lem: delta}.

(4) $\Rightarrow$ (1). Suppose that $\delta$ is an isomorphism. Then $\gamma$ is 1-1, so $N(\Op S)$ is spatial by Theorem~\ref{thm: spatial}.
Therefore, $Y_{\Op S}$ is weakly scattered by Theorem~\ref{thm: weakly scattered}. If $y \in Y_{\Op S}$, then $\{y\} \in N(X_{\Op S})$, so
$\delta(\{y\}) \ne \varnothing$ since $\delta$ is an isomorphism.
This implies $y \in \varepsilon[S]$. Thus, $\varepsilon$ is onto, and so $Y_{\Op S}$ is homeomorphic to $S_0$.
Consequently, $S_0$ is weakly scattered, and hence $S$ is weakly scattered by Proposition~\ref{prop: dispersed}(1).
\end{proof}

We can now recover one of the main results of \cite{Sim80}. For this we recall \cite[Def.~3.2]{Sim80} that $\sigma : N(\Op S) \to \Op_F(S)$
is defined by $\sigma(j)=\bigcup\{j(U)\setminus U\mid U\in \Op S\}$ for each $j\in N(\Op S)$. By \cite[Thm.~3.6]{Sim80}, $\sigma$
is an onto frame homomorphism.

\begin{corollary} \cite[Thm.~4.4]{Sim80}
A topological space $S$ is weakly scattered iff $\sigma : N(\Op S) \to \Op_F(S)$ is 1-1 (and hence an isomorphism).
\end{corollary}

\begin{proof}
By Theorem~\ref{prop: delta} it suffices to prove that $\sigma$ is 1-1 iff $\delta$ is 1-1.
Let $j \in N(\Op S)$ and let
$N_j$ be the corresponding nuclear set. We show $\sigma(j) = S \setminus \delta(N_j)$. Let $U \in \Op S$. We have
\[
N_j \subseteq (N_j \setminus \varphi(U)) \cup \varphi(U) \subseteq \down(N_j \setminus \varphi(U)) \cup \varphi(U),
\]
so
\[
[X_{\Op S} \setminus \down(N_j \setminus \varphi(U)] \setminus \varphi(U) \subseteq X_{\Op S} \setminus N_j.
\]
Recalling that $j$ satisfies
\[
\varphi(j(U)) = X_{\Op S} \setminus \down(N_j \setminus \varphi(U)),
\]
the last inclusion implies that $\varphi(j(U)) \setminus \varphi(U) \subseteq X_{\Op S} \setminus N_j$, so
$\delta(\varphi(j(U)) \setminus \varphi(U)) \subseteq \delta(X_{\Op S}\setminus N_j)$.
Therefore, $j(U) \setminus U \subseteq S \setminus \delta(N_j)$. Thus, $\sigma(j) \subseteq S \setminus \delta(N_j)$.

For the reverse inclusion,  let $s \in S \setminus \delta(N_j)$. Then $\varepsilon(s) \notin N_j$. Therefore, there are $U,V \in \Op S$ with
$\varepsilon(s) \in \varphi(U) \setminus \varphi(V)$ and $N_j \cap (\varphi(U) \setminus \varphi(V)) = \varnothing$. Thus,
$\varepsilon(s) \notin \varphi(V)$ and $(N_j \setminus \varphi(V)) \cap \varphi(U) = \varnothing$. Since $\varphi(U)$ is an upset,
$\varphi(U) \cap \down (N \setminus \varphi(V)) = \varnothing$, which implies $\varepsilon(s) \notin \down (N_j \setminus \varphi(V))$
because $\varepsilon(s) \in \varphi(U)$. This yields $\varepsilon(s) \in \varphi(j(V)) \setminus \varphi(V)$, so $s \in j(V) \setminus V$.
Consequently, $S \setminus \delta(N_j) \subseteq \sigma(j)$, and hence $\sigma(j) = S \setminus N_j$. From this
it follows that $\sigma$ is 1-1 iff $\delta$ is 1-1, completing the proof.
\end{proof}

We conclude the paper by determining when $N(\Op S)$ is boolean, and recovering another main result of Simmons \cite[Thm.~4.5]{Sim80}.

\begin{theorem} \label{thm: scat}
For a spatial frame $L$, the following conditions are equivalent.
\begin{enumerate}
\item $N(L)$ is boolean.
\item $N(L)$ is a complete and atomic boolean algebra.
\item $(Y_L,\tau)$ is scattered.
\end{enumerate}
\end{theorem}

\begin{proof}
(1) $\Rightarrow$ (2). Since $N(L)$ is boolean, $N(X_L) = {\sf RC}(X_L)$ (see Theorem~\ref{thm: boolean}). To see that $N(L)$ is atomic,
let $N \in N(X_L)$ be nonempty. As $N$ is regular closed, $N = \cl_\pi(\Int_\pi(N))$. Because $L$ is spatial, $Y_L$ is dense in $X_L$ by
Theorem~\ref{prop: dense = spatial}. Therefore, $N = \cl_\pi(\Int_\pi(N) \cap Y_L)$. Since $N$ is nonempty, so is $\Int_\pi(N) \cap Y_L$.
Thus, each $y \in \Int_\pi(N) \cap Y_L$ gives rise to the atom $\{y\}$ of $N(X_L)$ underneath $N$. This yields that $N(X_L)$ is an atomic
boolean algebra. Consequently, $N(L)$ is a complete and atomic boolean algebra.

(2) $\Rightarrow$ (3).
Let $F \ne \varnothing$ be a closed subset of $(Y_L,\tau)$. By Lemma~\ref{lem: tau}(2), $F = D \cap Y_L$ for some clopen downset $D$
of $(X_L,\pi)$. Since $N(L)$ is boolean, $\max(D)$ is clopen by Theorem~\ref{thm: boolean}. As $\max(D)$ is nonempty,
$\max(D) \cap Y_L \ne \varnothing$ because $Y_L$ is dense in $X_L$ by Theorem~\ref{prop: dense = spatial}.
Let $y \in \max(D) \cap Y_L$, and set $U=X_L \setminus (D\setminus\{y\})$. Since $y \in Y_L$, the singleton $\{y\}$ is nuclear, hence
regular closed because $N(L)$ is boolean. This implies $y$ is an isolated point of $X_L$, so $U$ is clopen, and it is an upset as
$y\in\max(D)$ and $D$ is a downset. Clearly
$U \cap D = \{y\}$. Thus, $U \cap D \cap Y_L = \{y\}$, and so $y$ is an isolated point of $F = D \cap Y_L$, proving that $Y_L$ is scattered.

(3) $\Rightarrow$ (1).
It is sufficient to show that $N(X_L)=\RC(X_L)$. Let $N \in N(X_L)$. Then $\cl_\pi(N\cap Y_L)$ is nuclear by Lemma~\ref{lem: nuclear}
and is contained in $N$. We have $N \cap Y_L = \cl_\pi(N\cap Y_L) \cap Y_L$. By Theorem~\ref{prop: delta}, $\delta$ is 1-1, so $\gamma$ is 1-1, which yields $N = \cl_\pi(N\cap Y_L)$. We show that each $y \in Y_L$ is isolated in $X_L$. Let $y \in Y_L$. Then $\down y \cap Y_L \in \F_\tau(Y_L)$. Since $(Y_L, \tau)$ is scattered, there is an isolated point $z$ of $\down y \cap Y_L$. By Lemma~\ref{lem: tau}(1), there is a clopen upset $U$ of $X_L$ with $\{z\} = U \cap \down y \cap Y_L$. Since $z \in \down y$ we have $z \le y$. Therefore, $y \in U$, and so $y \in U \cap \down y \cap Y_L$. This implies $z = y$. The set $U \cap \down y$ is clopen in $X_L$, so $(U \cap \down y) \setminus \{y\}$ is open in $X_L$. Since $Y_L$ is dense in $X_L$ and misses $(U \cap \down y) \setminus \{y\}$, it follows that $\{y\} = U \cap \down y$, hence $y$ is an isolated point of $X_L$. From this we conclude that $N \cap Y_L \subseteq \Int_\pi(N)$, so $N = \cl_\pi(\Int_\pi(N))$. Therefore, $N$ is regular closed. Thus, $N(X_L) \subseteq \RC(X_L)$, hence $N(X_L) = \RC(X_L)$ by Lemma~\ref{lem: basic nuclear facts}(2).
\end{proof}

\begin{corollary}
Let $L$ be a spatial frame. Then $L$ is scattered iff $Y_L$ is scattered.
\end{corollary}

\begin{proof}
Apply Theorems~\ref{thm: boolean} and~\ref{thm: scat}.
\end{proof}

As another consequence of Theorems~\ref{thm: boolean} and~\ref{thm: scat}, we obtain Simmons' theorem \cite{Sim80} for $T_0$-spaces.

\begin{corollary} \label{thm: scattered}
For a $T_0$-space $S$, the following are equivalent. 
\begin{enumerate}
\item $N(\Op S)$ is boolean.
\item $S$ is scattered.
\item $\Op S$ is scattered.
\end{enumerate}
\end{corollary}

\begin{proof}
(1) $\Rightarrow$ (2). First suppose that $N(\Op S)$ is boolean. By Theorem~\ref{thm: scat}, $(Y_{\Op S},\tau)$ is scattered. Since $S$ is a $T_0$-space,
$\varepsilon:S\to Y_{\Op S}$ is a topological embedding. Thus, $S$ is scattered.

(2) $\Rightarrow$ (1). Conversely, suppose $S$ is scattered. Since it is $T_0$, it is sober by Lemma~\ref{lem: weakly scattered implies sober}. Thus, $S$ is
homeomorphic to $(Y_{\Op S},\tau)$, and hence $(Y_{\Op S},\tau)$ is scattered. Applying Theorem~\ref{thm: scat} then yields that
$N(\Op S)$ is boolean.

(1) $\Leftrightarrow$ (3). Apply Theorem~\ref{thm: boolean}.
\end{proof}

Since $\Op S$ and $\Op (S_0)$ are isomorphic frames, as an immediate consequence of Corollary~\ref{thm: scattered} and
Proposition~\ref{prop: dispersed}(2), we arrive at the general form of Simmons' theorem.

\begin{corollary} \cite[Thm.~4.5]{Sim80}
For an arbitrary topological space $S$, the following are equivalent. 
\begin{enumerate}
\item $N(\Op S)$ is boolean.
\item $S$ is dispersed.
\item $\Op S$ is scattered.
\end{enumerate}
\end{corollary}

\def\cprime{$'$}
\providecommand{\bysame}{\leavevmode\hbox to3em{\hrulefill}\thinspace}
\providecommand{\MR}{\relax\ifhmode\unskip\space\fi MR }
\providecommand{\MRhref}[2]{%
  \href{http://www.ams.org/mathscinet-getitem?mr=#1}{#2}
}
\providecommand{\href}[2]{#2}

\bigskip

Departament of Mathematics, Universidad Aut\'{o}noma de Ciudad Ju\'{a}rez, fco.avila.mat@gmail.com

Department of Mathematical Sciences, New Mexico State University, Las Cruces NM 88003, guram@nmsu.edu

Department of Mathematical Sciences, New Mexico State University, Las Cruces NM 88003, pmorandi@nmsu.edu

Departament of Mathematics, University of Guadalajara, Blvd.\ Marcelino Garc{\'i}a Barrag\'{a}n, 44430, Guadalajara, Jalisco, M\'{e}xico, angelus31415@gmail.com


\begin{thebibliography}{10}

\bibitem{BM79}
R.~Beazer and D.~S. Macnab, \emph{Modal extensions of {H}eyting algebras},
  Colloq. Math. \textbf{41} (1979), no.~1, 1--12.

\bibitem{BB08}
G.~Bezhanishvili and N.~Bezhanishvili, \emph{Profinite {H}eyting algebras},
  Order \textbf{25} (2008), no.~3, 211--227.

\bibitem{BGJ13}
G.~Bezhanishvili, D.~Gabelaia, and M.~Jibladze, \emph{Funayama's theorem
  revisited}, Algebra Universalis \textbf{70} (2013), no.~3, 271--286.

\bibitem{BGJ16}
\bysame, \emph{Spectra of compact regular frames}, Theory Appl. Categ.
  \textbf{31} (2016), Paper No. 12, 365--383.

\bibitem{BG07}
G.~Bezhanishvili and S.~Ghilardi, \emph{An algebraic approach to subframe
  logics. {I}ntuitionistic case}, Ann. Pure Appl. Logic \textbf{147} (2007),
  no.~1-2, 84--100.

\bibitem{DP66}
C.~H. Dowker and D.~Papert, \emph{Quotient frames and subspaces}, Proc. London
  Math. Soc. (3) \textbf{16} (1966), 275--296.

\bibitem{Esa74}
L.~L. Esakia, \emph{Topological {K}ripke models}, Dokl. Akad. Nauk SSSR
  \textbf{214} (1974), 298--301.

\bibitem{Esa85}
\bysame, \emph{Heyting algebras. {I}. {D}uality theory}, ``Metsniereba'',
  Tbilisi, 1985 (Russian).

\bibitem{Isb72}
J.~Isbell, \emph{Atomless parts of spaces}, Math. Scand. \textbf{31} (1972),
  5--32.

\bibitem{Isb91}
\bysame, \emph{On dissolute spaces}, Topology Appl. \textbf{40} (1991), no.~1,
  63--70.

\bibitem{Joh82}
P.~T. Johnstone, \emph{Stone spaces}, Cambridge Studies in Advanced
  Mathematics, vol.~3, Cambridge University Press, Cambridge, 1982.

\bibitem{Mac81}
D.~S. Macnab, \emph{Modal operators on {H}eyting algebras}, Algebra Universalis
  \textbf{12} (1981), no.~1, 5--29.

\bibitem{NR87}
S.~B. Niefield and K.~I. Rosenthal, \emph{Spatial sublocales and essential
  primes}, Topology Appl. \textbf{26} (1987), no.~3, 263--269.

\bibitem{PP12}
J.~Picado and A.~Pultr, \emph{Frames and locales: Topology without points},
  Frontiers in Mathematics, Birkh\"auser/Springer Basel AG, Basel, 2012.

\bibitem{Ple00}
T.~Plewe, \emph{Higher order dissolutions and {B}oolean coreflections of
  locales}, vol. 154, 2000, Category theory and its applications (Montreal, QC,
  1997), pp.~273--293.

\bibitem{Ple02}
\bysame, \emph{Sublocale lattices}, J. Pure Appl. Algebra \textbf{168} (2002),
  no.~2-3, 309--326.

\bibitem{Pri70}
H.~A. Priestley, \emph{Representation of distributive lattices by means of
  ordered stone spaces}, Bull. London Math. Soc. \textbf{2} (1970), 186--190.

\bibitem{Pri72}
\bysame, \emph{Ordered topological spaces and the representation of
  distributive lattices}, Proc. London Math. Soc. (3) \textbf{24} (1972),
  507--530.

\bibitem{PS88}
A.~Pultr and J.~Sichler, \emph{Frames in {P}riestley's duality}, Cahiers
  Topologie G\'{e}om. Diff\'{e}rentielle Cat\'{e}g. \textbf{29} (1988), no.~3,
  193--202.

\bibitem{PS00}
\bysame, \emph{A {P}riestley view of spatialization of frames}, Cahiers
  Topologie G\'{e}om. Diff\'{e}rentielle Cat\'{e}g. \textbf{41} (2000), no.~3,
  225--238.

\bibitem{Sim78}
H.~Simmons, \emph{A framework for topology}, Logic {C}olloquium '77 ({P}roc.
  {C}onf., {W}roc\l aw, 1977), Stud. Logic Foundations Math., vol.~96,
  North-Holland, Amsterdam-New York, 1978, pp.~239--251.

\bibitem{Sim80}
\bysame, \emph{Spaces with {B}oolean assemblies}, Colloq. Math. \textbf{43}
  (1980), no.~1, 23--39 (1981).

\bibitem{Sim14}
\bysame, \emph{Cantor-{B}endixson properties of the assembly of a frame}, Leo
  {E}sakia on duality in modal and intuitionistic logics, Outst. Contrib. Log.,
  vol.~4, Springer, Dordrecht, 2014, pp.~217--255.

\bibitem{Wil94}
J.~T. Wilson, \emph{The assembly tower and some categorical and algebraic
  aspects of frame theory}, Ph.D. thesis, Carnegie Mellon University, 1994.

\end{thebibliography}
\end{document}